\documentclass[11pt,leqno,oneside,letterpaper]{amsart}

\usepackage[T1]{fontenc}
\usepackage[utf8]{inputenc}
\usepackage{lmodern}
\usepackage{eulervm}
\usepackage{amsmath,amssymb,amsthm}
\usepackage[margin=1.15in]{geometry}
\usepackage{booktabs,array}
\usepackage{enumitem}
\usepackage{xurl}
\usepackage[colorlinks,linkcolor=blue,citecolor=blue]{hyperref}
\usepackage[nameinlink,noabbrev]{cleveref}
\usepackage{orcidlink}
\newtheorem{theorem}{Theorem}[section]
\newtheorem{proposition}[theorem]{Proposition}
\newtheorem{lemma}[theorem]{Lemma}
\newtheorem{corollary}[theorem]{Corollary}

\theoremstyle{remark}
\newtheorem{remark}[theorem]{Remark}

\numberwithin{equation}{section}

\newcommand{\Q}{\mathbb Q}
\newcommand{\Z}{\mathbb Z}
\newcommand{\F}{\mathbb F}

\DeclareMathOperator{\Cl}{Cl}
\DeclareMathOperator{\Nm}{N}
\DeclareMathOperator{\ord}{ord}
\DeclareMathOperator{\rank}{rank}

\title[Ideal classes and rational points]{Ideal Classes of Order $2^n$ in a Family of Totally Real Cubic Fields and Rational Points on Associated Elliptic Curves}
\author{Junyu Lu\,\orcidlink{0009-0001-4973-3198}}
\address{Department of Mathematics, Yang-En University, Quanzhou, Fujian, China 362014}
\email{jy.lu@outlook.com}
\date{}

\begin{document}
\begin{abstract}
  For every prime $\ell\geq5$ and integer $n\geq2$, we construct a totally real cubic field in the family studied by Ennola with a specified prime ideal above $\ell$ whose class has exact order $2^n$. For parameters satisfying two square conditions and explicit congruences, the same squareclass obstructions give four independent rational points on the associated elliptic curves. A Pell equation describes all such parameters. A genus-one base change gives five independent sections that generate a subgroup saturated at two. It also yields infinitely many elliptic curves over $\Q$, pairwise nonisomorphic over $\overline{\Q}$, with rank at least five.
\end{abstract}
\maketitle

\section{Introduction}

We study a family of totally real cubic fields considered by Ennola in connection with exceptional units. Recall that a unit $\varepsilon$ is exceptional if $\varepsilon-1$ is also a unit \cite{Ennola1991,Louboutin2017}. For an integer $a\geq3$, let $\theta$ be the largest real root of
\[
  f_a(X)=X^3+(a-1)X^2-aX-1=X(X-1)(X+a)-1
\]
and put $K_a=\Q(\theta)$. This choice fixes the real embedding in which both $\theta$ and $\theta-1$ are positive. The polynomial $f_a(X)$ is irreducible, has three real roots, and has positive nonsquare discriminant $(a^2+3a-1)^2-32$. Thus $K_a$ is a totally real, non-Galois cubic field.

Ennola conjectured that $\theta$ and $\theta-1$ form a fundamental system of units of the maximal order; see \cite{Ennola1991,Louboutin2017,ChoiKim2025}. We require only the known description of the units of $K_a$ modulo squares, so our results do not depend on this conjecture. We write $\Cl(K_a)$ for the ordinary ideal class group of the maximal order $\mathcal O_{K_a}$. The same squareclasses control principal ideal powers and the independence of rational points on
\begin{equation}\label{eq:elliptic}
  E_a:\quad y^2=x(x+1)(x-a)+1.
\end{equation}
For $x\in\Q$, the connection is the identity
\begin{equation}\label{eq:common-norm}
  \Nm_{K_a/\Q}(x+\theta)=x(x+1)(x-a)+1.
\end{equation}

We use the following notation and conventions. For a number field $L$, $\Nm_{L/\Q}$ denotes the field norm. For a nonzero integral ideal $\mathfrak c\subset\mathcal O_L$, set $\Nm\mathfrak c=\#(\mathcal O_L/\mathfrak c)$, and extend this norm multiplicatively to nonzero fractional ideals. Brackets denote element squareclasses in $K_a^\times/K_a^{\times2}$ or ideal classes in $\Cl(K_a)$, according to the argument. Finite-place valuations are normalized to have value one on a uniformizer. The point at infinity is the identity $O$ of $E_a$.

Our first result prescribes both the rational prime and the exact order of a prime ideal class.

\begin{samepage}
  \begin{theorem}\label{thm:prescribed-prime}
    Let $\ell\geq5$ be a rational prime. For integers $n\geq2$, put $N=2^n$ and
    \[
      (D,a)=
      \begin{cases}
        \left(4,(\ell^N-73)/12\right),     & \ell\ne137, \\[2pt]
        \left(24,(137^N-14353)/552\right), & \ell=137.
      \end{cases}
    \]
    Then $a$ is an integer with $a\geq3$ and $a\equiv2\pmod4$. Moreover, the ideal
    \[
      \mathfrak p=(\ell,D\theta-D+1)\mathcal O_{K_a}
    \]
    is prime of norm $\ell$, and
    \[
      \mathfrak p^N=(D\theta-D+1),\qquad
      \ord_{\Cl(K_a)}[\mathfrak p]=N.
    \]
    For each fixed $\ell$, these parameters give infinitely many isomorphism classes of fields $K_a$ as $n$ varies.
  \end{theorem}
\end{samepage}

The change of linear form at $137$ is needed for the stated prime-ideal construction; see Remark~\ref{rem:137}. The order $\Z[\theta]$ need not be maximal. Proposition~\ref{prop:nonmaximal} gives an explicit infinite sequence where it is not.

\begin{theorem}\label{thm:four-points}
  Suppose that $a\geq3$ with $a\equiv30$ or $34\pmod{52}$. If there are positive integers $u,v$ satisfying
  \begin{equation}\label{eq:two-squares}
    u^2=12a+73,\qquad v^2=240a+4321,
  \end{equation}
  then the points
  \[
    P=(0,1),\quad Q=(-1,1),\quad
    S=(-3/4,u/8),\quad T=(-15/16,v/64)
  \]
  are independent modulo torsion in $E_a(\Q)$. The subgroup $H$ generated by $P,Q,S,T$ is saturated at two in the sense that $2R\in H$ for $R\in E_a(\Q)$ implies $R\in H$.

  If also $137\nmid u$ and $\gcd(v,23\cdot2503)=1$, then
  \[
    \mathfrak a=(u,4\theta-3)\mathcal O_{K_a},\qquad
    \mathfrak b=(v,16\theta-15)\mathcal O_{K_a}
  \]
  are ideals of the maximal order with
  \[
    \Nm\mathfrak a=u,\quad \mathfrak a^2=(4\theta-3),\qquad
    \Nm\mathfrak b=v,\quad \mathfrak b^2=(16\theta-15),
  \]
  and their classes generate a subgroup isomorphic to $(\Z/2\Z)^2$ in $\Cl(K_a)$.
\end{theorem}

The square conditions define $v^2-20u^2=2861$. Theorem~\ref{thm:pell} describes all positive solutions giving either required congruence on $a$, and Proposition~\ref{prop:subsequence} provides infinitely many parameters satisfying the ideal-class hypotheses. We also obtain an explicit family over $\Q(t)$ of rank at least four in Corollary~\ref{cor:rational-family}.

Proposition~\ref{prop:rank-five-rational} gives infinitely many rational parameters with rank at least five by a further genus-one base change. This construction has only one integral value of $a$, as shown in Remark~\ref{rem:rank-five-integrality}.

\subsection*{Relation to earlier work}

Ennola determined the unit squareclasses and constructed the unramified quadratic extension $K_a(\sqrt{\theta(\theta-1)})/K_a$ for $a\equiv2\pmod4$; he also obtained families with class-group $2$-rank at least two for sufficiently large admissible parameters \cite[Propositions 3.3 and 6.1, Theorem 6.1]{Ennola1991}. We recall Louboutin's short proof of the squareclass result \cite[Lemma 6]{Louboutin2017} and, under our ideal-class hypotheses, construct a second, linearly disjoint unramified quadratic extension in Corollary~\ref{cor:hilbert}.

Lee \cite[Theorem 1.1]{Lee2010} constructs ideal classes of any prescribed order in infinitely many fields of the same cubic family, written using $\rho=1/(\theta-1)$ and parameter $a+1$. Theorem~\ref{thm:prescribed-prime} prescribes the rational prime and gives explicit formulas for a prime ideal above it whose class has exact order $2^n$, without assuming $\mathcal O_{K_a}=\Z[\theta]$.

Washington \cite[\S3, Theorem 1]{Washington1987} related elliptic-curve ranks to ideal classes in simplest cubic fields; see also \cite{Duquesne2001,Laoudi2006}. For related elliptic curves, Walsh gives a Pell construction of rank at least three for sufficiently large parameters under auxiliary irreducibility hypotheses \cite[Theorem 1.1]{Walsh2023}, and a four-point construction with computational evidence of independence \cite[\S2, Type III]{Walsh2024}. We prove four-point independence and saturation at two under explicit square and congruence conditions, classify the resulting Pell parameters, and obtain five independent sections after a genus-one base change.

\section{Unit squareclasses and local obstructions}

The rational-root test proves that $f_a(X)$ is irreducible, since $f_a(1)=-1$ and $f_a(-1)=2a-3$. Its three roots lie respectively in $(-a,1-a)$, $(-1,0)$, and $(1,2)$, as follows by evaluation at the six endpoints. In this order of the real embeddings, both $\theta$ and $\theta-1$ have signature $(-,-,+)$ and norm one. Write
\begin{equation}\label{eq:discriminant}
  d_a=\operatorname{disc}(f_a)=(a^2+3a-1)^2-32.
\end{equation}

This integer is positive and odd. Since $\operatorname{disc}(K_a)=d_a/[\mathcal O_{K_a}:\Z[\theta]]^2$, the index is odd and $\mathcal O_{K_a}\otimes\Z_2=\Z_2[\theta]$. The Weierstrass equation \eqref{eq:elliptic} has discriminant $16d_a$, so it is nonsingular. Also $d_a$ lies strictly between the consecutive squares $(a^2+3a-2)^2$ and $(a^2+3a-1)^2$. The field discriminant is the quotient of $d_a$ by a square, so it too is nonsquare and the cubic field is non-Galois. See \cite{Ennola1991} for more details.

\begin{lemma}\label{lem:units}
  For every $a\geq3$, the positive-norm units of $K_a$, modulo squares of units, have the four distinct representatives
  \[
    1,\quad \theta,\quad \theta-1,\quad \theta(\theta-1).
  \]
\end{lemma}

\begin{proof}
  We recall the argument in \cite[Lemma 6]{Louboutin2017}. The signs exclude $\theta$ and $\theta-1$ from being squares. Their product has the same squareclass as
  \[
    \epsilon=\frac{\theta}{\theta-1}=\theta^2+a\theta+1.
  \]
  Since $\theta=\epsilon/(\epsilon-1)$, this unit generates $K_a$; its minimal polynomial is $Z^3-(a+4)Z^2+(a+3)Z-1$. If $\epsilon=\eta^2$ in $K_a$, then $\eta$ is an integral unit of degree three. Replacing $\eta$ by $-\eta$ if necessary, we may assume that its norm is one. Its minimal polynomial then has the form $Z^3-sZ^2+wZ-1$ with $s,w\in\Z$. Comparing the elementary symmetric functions of its squared conjugates gives
  \[
    s^2-2w=a+4,\qquad w^2-2s=a+3,
    \qquad (s-w)(s+w+2)=1.
  \]
  The integer pairs are $(s,w)=(0,-1),(-2,-1)$, giving $a=-2,2$, respectively. Neither is allowed. Thus $[\theta]$ and $[\theta-1]$ are independent in $K_a^\times/K_a^{\times2}$. A square root of a unit is a unit, by its valuations at finite primes. Dirichlet's unit theorem gives eight unit squareclasses, and $\Nm_{K_a/\Q}$ maps this group onto $\{1,-1\}$ because $\Nm_{K_a/\Q}(-1)=-1$. Its kernel therefore consists of the four displayed classes.
\end{proof}

\begin{lemma}\label{lem:obstruction}
  Suppose that $a\geq3$ and $a\equiv2\pmod4$. If $D$ is a positive integer with $4\mid D$, then $D\theta-D+1$ is not a unit times a square in $K_a$.
\end{lemma}

\begin{proof}
  The odd index gives
  \[
    \mathcal O_{K_a}/4\mathcal O_{K_a}\simeq
    (\Z/4\Z)[X]/(X^3+X^2+2X-1).
  \]
  Its reduction modulo two is $\F_8$. The unique square roots of $\theta$ and $\theta-1$ in $\F_8$ are $\theta^2+\theta+1$ and $\theta^2+\theta$. Lifts of the same residue modulo two have equal squares modulo four, whereas
  \begin{align*}
    (\theta^2+\theta+1)^2 & \equiv\theta+2\not\equiv\theta\pmod4,              \\
    (\theta^2+\theta)^2   & \equiv2\theta^2+3\theta+1\not\equiv\theta-1\pmod4.
  \end{align*}
  The second difference is $2(\theta^2+\theta+1)$, which is nonzero because $1,\theta,\theta^2$ form a $(\Z/4\Z)$-basis of this quotient. Thus neither $\theta$ nor $\theta-1$ is a square modulo four in the maximal order.

  Put $\gamma=D\theta-D+1$. Since $D\geq4$, the root intervals give its signature as $(-,-,+)$. Thus $\Nm_{K_a/\Q}(\gamma)>0$, and $\gamma\equiv1\pmod{4\mathcal O_{K_a}}$. The products $\gamma$ and $\gamma\theta(\theta-1)$ have negative conjugates. The products $\gamma\theta$ and $\gamma(\theta-1)$ are nonsquares modulo four. These tests exclude squares in $K_a$, since a square root of an algebraic integer is integral. If $\gamma=\varepsilon\xi^2$ for a unit $\varepsilon$, taking norms to $\Q$ gives $\Nm_{K_a/\Q}(\varepsilon)=1$. Lemma~\ref{lem:units} would then make one of the four excluded products a square.
\end{proof}

Put $\alpha=4\theta-3$ and $\beta=16\theta-15$. Their norms are
\begin{equation}\label{eq:alpha-beta}
  \Nm_{K_a/\Q}(\alpha)=12a+73,\qquad \Nm_{K_a/\Q}(\beta)=240a+4321.
\end{equation}

\begin{lemma}\label{lem:four-classes}
  If $a\geq3$ and $a\equiv30$ or $34\pmod{52}$, then the four classes $[\theta]$, $[\theta-1]$, $[\alpha]$, and $[\beta]$ are independent in $K_a^\times/K_a^{\times2}$.
\end{lemma}

\begin{proof}
  Both congruences give $a\equiv2\pmod4$. Lemma~\ref{lem:units} excludes the three nontrivial products of the unit classes, and Lemma~\ref{lem:obstruction} excludes the eight products involving exactly one of $\alpha,\beta$. The products $\alpha\beta\theta$ and $\alpha\beta(\theta-1)$ have negative conjugates. It remains to exclude $\alpha\beta$ and $\alpha\beta\theta(\theta-1)$.

  If $a\equiv30\pmod{52}$, then $a\equiv4\pmod{13}$ and
  \[
    f_a(X)\equiv(X-6)(X^2+9X+11)\pmod{13},\qquad d_a\equiv8\pmod{13}.
  \]
  The prime $\mathfrak l=(13,\theta-6)$ of the maximal order has residue field $\F_{13}$. At this prime,
  \[
    \alpha\beta\equiv11,\qquad \theta(\theta-1)\equiv4.
  \]
  The two remaining products reduce to $11$ and $5$, both nonsquares in $\F_{13}$.

  If $a\equiv34\pmod{52}$, then $a\equiv8\pmod{13}$ and
  \[
    f_a(X)\equiv(X-3)(X-4)(X-12)\pmod{13}.
  \]
  These distinct roots define degree-one primes of the maximal order. At the prime with $\theta\equiv3$, we have $\alpha\beta\equiv11$. At the prime with $\theta\equiv12$, we have $\alpha\beta\theta(\theta-1)\equiv5$. Both residues are nonsquares. Thus all fifteen nontrivial products are excluded in either case.
\end{proof}

\section{Prime ideal classes of prescribed order}

We first record the local step that allows us to work in the maximal order without assuming $\mathcal O_{K_a}=\Z[\theta]$.

\begin{lemma}\label{lem:local-power}
  Let $f\in\Z[X]$ be monic and irreducible, let $K=\Q(\theta)$ with $f(\theta)=0$, and let $D,c,m,N$ be integers with $D,m>0$, $N\geq1$, and $\gcd(D,m)=1$. If $\Nm_{K/\Q}(D\theta-c)=m^N$ and $f'(c/D)$ is a $p$-adic unit for every prime $p\mid m$, then
  \[
    I=(m,D\theta-c)\mathcal O_K
    \quad\text{satisfies}\quad \Nm I=m,\qquad I^N=(D\theta-c).
  \]
  If $m$ is a rational prime, $I$ is prime.
\end{lemma}

\begin{proof}
  Fix $p\mid m$. The norm identity implies $f(c/D)\equiv0\pmod p$. The derivative hypothesis and Hensel's lemma give $f(X)=(X-r)h(X)$ in $\Z_p[X]$, with $r\equiv c/D\pmod p$ and $h(r)\in\Z_p^\times$. The coprime factors give $\Z_p[X]/(f)\simeq\Z_p\times\Z_p[X]/(h)$. Taking integral closures in $K\otimes_{\Q}\Q_p$ preserves this product. The first factor is already integrally closed and equals $\Z_p$, so the corresponding prime $\mathfrak p_p$ of $K$ has ramification index and residue degree both one. In the complementary factor, $DX-c$ is a unit modulo $p$ because $h(c/D)$ is a $p$-adic unit. Its image is therefore a unit in the integral closure. Thus $\mathfrak p_p$ is the only prime above $p$ that divides $(D\theta-c)$.

  The norm identity excludes prime divisors of $(D\theta-c)$ above rational primes not dividing $m$. Comparing valuations of the norm therefore gives
  \[
    (D\theta-c)=\prod_{p\mid m}\mathfrak p_p^{Nv_p(m)}.
  \]
  Taking the ideal sum with $(m)$ replaces each exponent by $v_p(m)$. This proves both identities and the last assertion.
\end{proof}

For later use, direct evaluation gives
\begin{equation}\label{eq:general-norm}
  \Nm_{K_a/\Q}(D\theta-D+1)=D(D-1)a+D^3+D^2-2D+1.
\end{equation}

\begin{proof}[Proof of Theorem~\ref{thm:prescribed-prime}]
  First let $\ell\ne137$ and $D=4$. Because $N$ is divisible by four and $\ell$ is coprime to six, $\ell^N\equiv1\pmod{48}$. Thus $a=(\ell^N-73)/12$ is an integer with $a\geq46$ and $a\equiv2\pmod4$. Equations~\eqref{eq:alpha-beta} and
  \[
    f_a'(3/4)=\frac{2\ell^N-137}{48}
  \]
  give the hypotheses of Lemma~\ref{lem:local-power}. Hence the indicated ideal $\mathfrak p$ has norm $\ell$ and satisfies $\mathfrak p^N=(4\theta-3)$.

  For $\ell=137$ and $D=24$, the congruence $137^4\equiv1\pmod{2208}$ gives
  \[
    a=(137^N-14353)/552\in\Z,\qquad a\equiv2\pmod4,
  \]
  and $a\geq638154$. Equation~\eqref{eq:general-norm} and differentiation give
  \[
    \Nm_{K_a/\Q}(24\theta-23)=137^N,\qquad
    f_a'(23/24)=\frac{22\cdot137^N-304657}{13248}.
  \]
  The denominator is prime to $137$, and $304657\equiv106\pmod{137}$. Lemma~\ref{lem:local-power} again gives a prime of norm $137$ with the required power identity.

  In either case, $[\mathfrak p]^N=1$. If its order were a proper divisor of the power of two $N$, then $\mathfrak p^{N/2}$ would be principal. Writing $\mathfrak p^{N/2}=(\xi)$ would make $D\theta-D+1$ a unit times $\xi^2$, contrary to Lemma~\ref{lem:obstruction}. Hence $[\mathfrak p]$ has exact order $N$. Finally, a finite set of number fields has class groups of bounded exponent. The unbounded orders $2^n$ prove the assertion about field isomorphism classes.
\end{proof}

For comparison, suppose that $\ell\ne137$ and $\mathcal O_{K_a}=\Z[\theta]$. If $\mathfrak p^{N/2}=(\gamma)$, then $\gamma$ is integral and
\[
  |\Nm_{K_a/\Q}(\gamma)|=\ell^{N/2}=\sqrt{12a+73}<2a-3,
\]
since $a\geq46$ and $(2a-3)^2-(12a+73)=4(a-8)(a+2)>0$. The norm bound in \cite{KalaSgallovaTinkova2025} forces $\gamma$ to be a unit times a rational integer. Its absolute norm would then be a cube, contrary to $3\nmid N/2=2^{n-1}$.

Louboutin's bound \cite[Theorem 10]{Louboutin2001} gives the same deduction for $a\geq59$, with his parameter $m=a+1$ and bound $2m-5=2a-3$. The identities $\rho=1/(\theta-1)=\theta(\theta+a)$ and $\theta=\rho^2-(a+1)\rho-(a+1)$ give $\Z[\rho]=\Z[\theta]$, so the maximality hypothesis agrees for the two generators $\rho$ and $\theta$.

\begin{corollary}\label{cor:composite}
  Suppose $n\geq2$, $N=2^n$, and $m\geq5$ with $\gcd(m,6\cdot137)=1$. If $a=(m^N-73)/12$, then the ideal $I=(m,4\theta-3)\mathcal O_{K_a}$ has norm $m$, satisfies $I^N=(4\theta-3)$, and has exact class order $N$.
\end{corollary}

\begin{proof}
  The fourth-power congruence modulo $48$ holds for this $m$, and the derivative calculation in the first part of the preceding proof applies at every prime divisor of $m$. Apply Lemmas~\ref{lem:local-power} and \ref{lem:obstruction}, and the same exact-order argument.
\end{proof}

\begin{remark}\label{rem:137}
  The first formula cannot be extended to $\ell=137$ with the same ideal. Put $q=137^N$, $a=(q-73)/12$, and $\alpha=4\theta-3$. The minimal polynomial of $\alpha$ is
  \[
    Z^3+\frac{q-58}{3}Z^2+\frac{2q-137}{3}Z-q.
  \]
  Write this polynomial as $G(Z)$ and put $p=137$. The following three rescaled polynomials have reductions
  \begin{align*}
    G(p^{N-1}Z)/p^N & \equiv -Z/3-1,              \\
    G(pZ)/p^2       & \equiv -Z(58Z+1)/3,         \\
    G(Z)            & \equiv Z^2(Z-58/3) \pmod p.
  \end{align*}
  Their respective nonzero roots $-3$, $-1/58$, and $58/3$ are simple.

  Hensel's lemma gives three roots of $G$ in $\Q_{137}$ of valuations $N-1$, $1$, and $0$. Thus $137$ splits completely, $(\alpha)=\mathfrak p^{N-1}\mathfrak q$ for two distinct primes of norm $137$, and $(137,\alpha)$ has norm $137^2$.

  For the second parameter formula $a=(q-14353)/552$ in Theorem~\ref{thm:prescribed-prime}, the form $24\theta-23$ avoids this obstruction. At $x=-23/24$, its norm identity gives $y^2=137^N/24^3$, which is not a rational square. Thus it does not give a rational point on \eqref{eq:elliptic} by the denominator-clearing construction used for $4\theta-3$.
\end{remark}

\begin{proposition}\label{prop:nonmaximal}
  Let $n\geq2$ and $a_n=(5^{2^n}-73)/12$. If $n\equiv13\pmod{110}$, then
  \[
    23\mid[\mathcal O_{K_{a_n}}:\Z[\theta]].
  \]
  These parameters include infinitely many isomorphism classes of fields.
\end{proposition}

\begin{proof}
  The modular identities
  \[
    5^{506}\equiv1\pmod{6348},\quad
    2^{110}\equiv1\pmod{253},\quad
    5^{2^{13}}\equiv2845\pmod{6348}
  \]
  give $a_n\equiv231\pmod{529}$. Indeed, $2^n\equiv2^{13}\pmod{253}$, and both exponents are even, giving the congruence modulo $506$. Write $a_n=231+529k$. Substitution using $f_{a_n}(\theta)=0$ shows that $\xi=(\theta^2-13\theta+30)/23$ satisfies
  \begin{align*}
    \xi^3 & -(12167k^2+10925k+2454)\xi^2                    \\
          & +(25392k^2+22802k+5119)\xi-(138k+61)(90k+41)=0.
  \end{align*}
  This monic equation makes $\xi$ integral, while the basis $1,\theta,\theta^2$ shows that $\xi\notin\Z[\theta]$ and $23\xi\in\Z[\theta]$. This proves the index assertion. Theorem~\ref{thm:prescribed-prime} gives unbounded class orders along the progression, so finitely many fields cannot account for it.
\end{proof}

\section{Four points and two ideal classes}

\begin{lemma}\label{lem:virtual-units}
  Let $a\geq3$, and let
  \[
    E_{K_{a}}^{(1)}=\ker(\Nm_{{K_{a}}/\Q}:\mathcal O_{K_{a}}^\times\to\{\pm1\})
  \]
  be its group of norm-one units. Define
  \[
    V_{K_{a}}=\{\gamma\in {K_{a}}^\times:\Nm_{{K_{a}}/\Q}(\gamma)>0,\ v_{\mathfrak p}(\gamma)\in2\Z\ \text{for every finite prime }\mathfrak p\}/{K_{a}}^{\times2}.
  \]
  Then there is an exact sequence
  \[
    1\longrightarrow E_{K_{a}}^{(1)}/(\mathcal O_{K_{a}}^\times)^2
    \longrightarrow V_{K_{a}}\longrightarrow\Cl({K_{a}})[2]\longrightarrow1.
  \]
  The kernel of $V_{K_{a}}\to\Cl({K_{a}})[2]$ has dimension two over $\F_2$, with basis $[\theta],[\theta-1]$.
\end{lemma}

\begin{proof}
  The even valuations determine a unique fractional ideal $\mathfrak c$ with $(\gamma)=\mathfrak c^2$. Replacing $\gamma$ by a square multiple changes $\mathfrak c$ by a principal ideal, so the map is well defined. Its kernel consists of the classes of norm-one units: if $\mathfrak c=(\xi)$, then $\gamma/\xi^2$ is such a unit. Conversely, each norm-one unit represents an element of the kernel. A square root in ${K_{a}}$ of a unit is a unit, so the left map is injective. If $[\mathfrak c]\in\Cl({K_{a}})[2]$, write $\mathfrak c^2=(\gamma)$; replacing $\gamma$ by $-\gamma$ when necessary makes its norm positive because $[{K_{a}}:\Q]=3$. This proves surjectivity. The kernel assertion follows from Lemma~\ref{lem:units}.
\end{proof}

\begin{proof}[Proof of Theorem~\ref{thm:four-points}]
  Define
  \[
    \delta:E_a(\Q)\longrightarrow K_a^\times/K_a^{\times2},
    \qquad\delta(O)=1,\quad\delta((x,y))=[x+\theta].
  \]
  Irreducibility ensures that $x+\theta\ne0$. The map is a homomorphism: if a nonvertical line $y=mX+b$ meets the cubic at abscissae $x_1,x_2,x_3$, with multiplicity, evaluation of the monic intersection identity at $X=-\theta$ gives
  \[
    (x_1+\theta)(x_2+\theta)(x_3+\theta)=(b-m\theta)^2.
  \]
  The chord-and-tangent law and the vertical-line case prove the assertion. In particular, $\delta$ annihilates doubles. There is no rational two-torsion because the defining cubic is irreducible. The rational torsion group has odd order, so $\delta$ also annihilates torsion.

  The square conditions verify that $S,T$ lie on the curve. The images of $P,Q,S,T$ are $[\theta]$, $[\theta-1]$, $[\alpha]$, and $[\beta]$, independent by Lemma~\ref{lem:four-classes}. Suppose that an integral combination of these points, with coefficients not all zero, is torsion. Let $d$ be the greatest common divisor of the coefficients, and let $R$ be the combination obtained by dividing them by $d$. Since $dR$ is torsion, $R$ is torsion as well. The resulting relation is therefore primitive. Applying $\delta$ forces every coefficient to be even, contradicting primitivity. Thus the four points are independent modulo torsion.

  Now suppose that $2R=bP+cQ+eS+fT$. Applying $\delta$ makes all four coefficients even. Subtracting $(b/2)P+(c/2)Q+(e/2)S+(f/2)T$ from $R$ leaves rational two-torsion and therefore gives $O$. This proves saturation at two.

  For the ideals, the square conditions give $\gcd(u,6)=1$ and $\gcd(v,30)=1$. The derivatives are
  \[
    f_a'(3/4)=\frac{2u^2-137}{48},\qquad
    f_a'(15/16)=\frac{14v^2-57569}{3840},
    \quad 57569=23\cdot2503.
  \]
  The additional coprimality conditions and Lemma~\ref{lem:local-power}, with exponent two, prove the norm and square identities for $\mathfrak a,\mathfrak b$.

  The ideal identities place $[\alpha],[\beta]$ in $V_{K_a}$. If $[\mathfrak a]^e[\mathfrak b]^f=1$ for $e,f\in\{0,1\}$, Lemma~\ref{lem:virtual-units} puts $[\alpha^e\beta^f]$ in the span of $[\theta],[\theta-1]$. Lemma~\ref{lem:four-classes} forces $e=f=0$. Thus the four independent squareclasses yield two independent classes in $\Cl({K_a})[2]$ after quotienting by the two unit classes.
\end{proof}

\begin{samepage}
  \begin{corollary}\label{cor:hilbert}
    Under the ideal-class hypotheses of Theorem~\ref{thm:four-points}, the extension
    \[
      K_a\bigl(\sqrt{\theta(\theta-1)},\sqrt{(4\theta-3)(16\theta-15)}\bigr)/K_a
    \]
    is biquadratic and unramified at every finite and infinite prime. It is contained in the ordinary Hilbert class field.
  \end{corollary}
\end{samepage}

The proof uses signatures and congruences modulo four, as in \cite[\S I.6(b--c) and Proposition I.5]{GrasGras1975}; we give the local argument for the present non-Galois fields.

\begin{proof}
  Let $h_1=\theta(\theta-1)$ and $h_2=\alpha\beta$. Both are totally positive and have even valuations at every finite prime, by the ideal identities. Hence at odd primes their quadratic extensions are unramified or split. At primes above two both are units, and
  \[
    h_1\equiv(\theta^2+1)^2\pmod{4\mathcal O_{K_a}},\qquad
    h_2\equiv1\pmod{4\mathcal O_{K_a}}.
  \]
  For a unit $h\equiv c^2\pmod4$, the polynomial $Z^2-cZ+(c^2-h)/4$ is integral and has unit discriminant $h$; its quadratic algebra is therefore unramified or split at a dyadic prime. At every real place, total positivity makes both quadratic extensions split. Finally, Lemma~\ref{lem:four-classes} makes $[h_1]$ and $[h_2]$ independent, so the extension has degree four.
\end{proof}

\begin{samepage}
  \begin{corollary}\label{cor:three-points}
    For the parameters in Corollary~\ref{cor:composite}, the points
    \[
      (0,1),\qquad (-1,1),\qquad (-3/4,m^{N/2}/8)
    \]
    on $E_a$ are independent modulo torsion and generate a subgroup saturated at two.
  \end{corollary}
\end{samepage}

\begin{proof}
  The norm identity verifies the third point. Lemmas~\ref{lem:units} and \ref{lem:obstruction} make their three images under $\delta$ independent. The arguments proving independence modulo torsion and saturation at two in the proof of Theorem~\ref{thm:four-points} apply.
\end{proof}

\section{Pell parameters and a rational family}

\begin{samepage}
  \begin{theorem}\label{thm:pell}
    Put $\eta=9+2\sqrt{20}$. All positive integer pairs $(u,v)$ for which \eqref{eq:two-squares} holds with an integer $a\geq3$ satisfying $a\equiv30$ or $34\pmod{52}$ are given by
    \begin{align*}
      v+u\sqrt{20} & =(79+13\sqrt{20})\eta^k,
                   & k                         & \geq0,\quad k\equiv6,7,14,15,22,27\pmod{28}, \\
      v+u\sqrt{20} & =(191+41\sqrt{20})\eta^k,
                   & k                         & \geq0,\quad k\equiv0,5,12,13,20,21\pmod{28},
    \end{align*}
    with $a=(u^2-73)/12$. There are infinitely many such parameters and infinitely many geometric isomorphism classes of the associated elliptic curves.
  \end{theorem}
\end{samepage}

\begin{proof}
  For a positive solution of $v^2-20u^2=2861$, set $z=v+u\sqrt{20}$. There is a unique integer $k\geq0$ with
  \[
    \sqrt{2861}\leq z\eta^{-k}<\eta\sqrt{2861}.
  \]
  Write $z_0=z\eta^{-k}=v_0+u_0\sqrt{20}$. Both coordinates are integers because $\eta^{-1}=9-2\sqrt{20}$. The conjugate of $z_0$ is $2861/z_0$. Hence $u_0=(z_0-2861/z_0)/(2\sqrt{20})$, and the reduction interval, together with $\eta-\eta^{-1}=4\sqrt{20}$, gives
  \[
    v_0>0,\qquad 0\leq u_0<2\sqrt{2861}<107.
  \]
  Testing $v_0^2=20u_0^2+2861$ modulo $16,3,7,11$ leaves eight candidates in the interval $0\leq u_0\leq106$: $u_0=13,19,41,47,55,79,85,97$. Modulo thirteen only $13,41,55$ remain. Their candidate squares are $79^2$, $191^2$, and $63361$, respectively; the last lies strictly between $251^2$ and $252^2$. Both remaining initial solutions lie in the reduction interval. Thus multiplication of $79+13\sqrt{20}$ and $191+41\sqrt{20}$ by $\eta^k$, $k\geq0$, gives all positive solutions before the congruence conditions on $a$ are imposed.

  In the coordinate order $(v,u)$, multiplication by $\eta$ gives the recurrence
  \begin{equation}\label{eq:recurrence}
    \binom{v_{k+1}}{u_{k+1}}=
    M\binom{v_k}{u_k},\qquad
    M=\begin{pmatrix}9&40\\2&9\end{pmatrix}.
  \end{equation}
  The congruences $a\equiv30\pmod{52}$ and $a\equiv34\pmod{52}$ correspond respectively to $u^2\equiv433\pmod{624}$ and $u^2\equiv481\pmod{624}$. Each initial solution has period $56$ under \eqref{eq:recurrence} modulo $624$; the complete calculation gives Table~\ref{tab:pell}. The matrix has determinant one, and each initial ordered pair first returns after $56$ applications of $M$. Thus the table is exhaustive. The exponents satisfying either congruence repeat modulo $28$, giving the stated lists. The smallest admissible parameter is $a=134$, corresponding to $(v,u)=(191,41)$.

  \begin{table}[ht]
    \centering

    \begin{tabular}{ccc}
      \toprule
      Initial solution $(v_0,u_0)$ & $a\bmod52$ & Admissible exponents $k$ \\
      \midrule
      $(79,13)$                    & $30$       & $6,15,22,27,34,43,50,55$ \\
      $(79,13)$                    & $34$       & $7,14,35,42$             \\
      $(191,41)$                   & $30$       & $0,5,12,21,28,33,40,49$  \\
      $(191,41)$                   & $34$       & $13,20,41,48$            \\
      \bottomrule
    \end{tabular}
    \caption{Complete congruence check for the two initial Pell solutions modulo $624$. The listed exponents lie in $0\leq k<56$.}
    \label{tab:pell}
  \end{table}

  Both positive coordinates grow strictly along either orbit. Finally,
  \[
    j(E_a)=\frac{256(a^2+a+1)^3}{d_a}\sim256a^2
    \quad(a\longrightarrow\infty).
  \]
  The resulting $j$-invariants are unbounded, proving the assertion about geometric isomorphism classes.
\end{proof}

\begin{proposition}\label{prop:subsequence}
  In the orbit starting from $(v_0,u_0)=(191,41)$, take $k=28r$ with $r\geq0$. The ideal-class hypotheses of Theorem~\ref{thm:four-points} hold exactly when $r\not\equiv63\pmod{313}$. These parameters give infinitely many isomorphism classes of cubic fields.
\end{proposition}

\begin{proof}
  Theorem~\ref{thm:pell} gives the hypotheses for the four-point conclusion. The remaining congruences follow from the complete periods in Table~\ref{tab:coprime}, obtained from \eqref{eq:recurrence}. In each row the initial ordered pair first returns at the displayed period; invertibility of $M$ modulo the modulus excludes a nonperiodic initial segment. The two residues in the last row say $2503\mid v_k$ exactly when $k\equiv512\pmod{1252}$. For $k=28r$ this becomes $7r\equiv128\pmod{313}$, or $r\equiv63\pmod{313}$.

  \begin{table}[ht]
    \centering

    \begin{tabular}{cccc}
      \toprule
      Modulus & Period & Coordinate & Vanishing indices $k$ \\
      \midrule
      $137$   & $46$   & $u_k$      & none                  \\
      $23$    & $8$    & $v_k$      & none                  \\
      $2503$  & $2504$ & $v_k$      & $512,1764$            \\
      \bottomrule
    \end{tabular}
    \caption{Complete coprimality checks on the orbit from $(191,41)$. The last column lists the indices $k$ within one period at which the indicated coordinate vanishes modulo the modulus.}
    \label{tab:coprime}
  \end{table}

  For modulus $2503$, the period and the vanishing indices follow from explicit identities. With $w=(191,41)^{\mathsf T}$, calculation modulo $2503$ gives
  \begin{gather*}
    M^8w=(52,993)^{\mathsf T},\qquad M^{512}w=(0,1022)^{\mathsf T},\\
    M^{1252}w=-w,\qquad M^{1764}w=(0,1481)^{\mathsf T},\qquad M^{2504}w=w.
  \end{gather*}
  Since $2504=8\cdot313$ and $313$ is prime, every proper divisor of $2504$ divides either $8$ or $1252$. The first and third identities therefore prove the period. The invariant $v^2-20u^2=2861$ has at most two residue pairs with $v=0$ modulo the prime $2503$; the displayed pairs give both. This proves that the listed vanishing indices are exhaustive.

  The equation $x+y=1$ has only finitely many solutions in units of a fixed number field \cite[Theorem IX.4.1]{Silverman2009}. If infinitely many of these fields were isomorphic to one fixed field $K$, the images $x$ of $\theta$ would give infinitely many units for which $1-x$ is a unit. They are distinct for distinct parameters because
  \[
    a=\frac{1-x^3+x^2}{x^2-x},\qquad x\ne0,1.
  \]
  This is a contradiction, and a finite collection of fields cannot account for the infinite sequence.
\end{proof}

We now allow $a$ to be a rational function on a curve over $\Q$. The restrictions $a\in\Z$ and $a\geq3$ no longer apply to the remaining constructions.

\begin{corollary}\label{cor:rational-family}
  For an indeterminate $t$, put
  \begin{align*}
    u(t) & =\frac{41t^2-382t+820}{20-t^2},    \\
    v(t) & =\frac{191t^2-1640t+3820}{20-t^2}, \\
    a(t) & =\frac{u(t)^2-73}{12}.
  \end{align*}
  The four sections
  \[
    (0,1),\quad(-1,1),\quad(-3/4,u(t)/8),\quad(-15/16,v(t)/64)
  \]
  are independent modulo torsion in $E_{a(t)}(\Q(t))$ and generate a subgroup saturated at two. In particular, the rank over $\Q(t)$ is at least four.
\end{corollary}

\begin{proof}
  Intersecting the line $v=191-t(u+41)$ through $(u,v)=(-41,191)$ with $v^2-20u^2=2861$ gives the displayed rational parametrization at its second intersection. Direct substitution verifies the equation, and $u(0)=41$, $v(0)=191$, $a(0)=134$. The projective Weierstrass equation has coefficients in $\Q[t]_{(t)}$ and unit discriminant there, so it defines an elliptic scheme over this local ring. By properness, every $\Q(t)$-point extends over the local ring, and the group law makes specialization a homomorphism. A nontrivial integral relation modulo torsion therefore specializes to such a relation on $E_{134}$, contrary to Theorem~\ref{thm:four-points}. This proves independence.

  There is no nonzero two-torsion over $\Q(t)$. Otherwise $f_{a(t)}$ would have a root in $\Q(t)$. Since $a(t)$ is regular at zero, this polynomial is monic over the integrally closed local ring $\Q[t]_{(t)}$. The root would belong to that ring and reduce to a rational root of $f_{134}$, contradicting irreducibility.

  If twice a rational section is an integral combination of the four sections, specialize at zero and apply the squareclass independence on $E_{134}$. All four coefficients are even. Subtracting the combination with these halved coefficients from the section leaves rational two-torsion over $\Q(t)$, hence zero. This proves saturation at two.
\end{proof}


\section{Infinitely many rational parameters with rank at least five}

To obtain a fifth point, impose $w^2=577-72a$. Together with the two square conditions in \eqref{eq:two-squares}, this defines the curve $C\subset\mathbb P^3_{\Q}$ with homogeneous coordinates $[U:V:W:Z]$ and equations
\begin{equation}\label{eq:rank-five-base}
  V^2-20U^2=2861Z^2,\qquad W^2+6U^2=1015Z^2.
\end{equation}
On $Z\ne0$ put $u=U/Z$, $v=V/Z$, $w=W/Z$, and $a=(u^2-73)/12$. The second equation gives $w^2=577-72a$, so $(8,w)$ lies on $E_a$.

\begin{proposition}\label{prop:rank-five-rational}
  The curve $C$ is a smooth genus-one curve with infinitely many rational points. Over $\Q(C)$ the five points
  \[
    (0,1),\quad(-1,1),\quad(-3/4,u/8),\quad(-15/16,v/64),\quad(8,w)
  \]
  are independent modulo torsion and generate a subgroup saturated at two. Consequently, infinitely many geometrically nonisomorphic curves $E_a$, with $a\in\Q$, satisfy both square conditions in \eqref{eq:two-squares} and have rank at least five.
\end{proposition}

\begin{proof}
  The pencil of quadrics generated by the two equations has determinant proportional to
  \[
    \lambda\mu(-20\lambda+6\mu)(-2861\lambda-1015\mu).
  \]
  This homogeneous polynomial has four distinct zeros $[\lambda:\mu]\in\mathbb P^1$. Each singular quadric in this diagonal pencil has its vertex at a coordinate point, and none of these points lies on the intersection. The Jacobian criterion therefore shows that the intersection is smooth. As a complete intersection of two quadrics in $\mathbb P^3$, it is geometrically connected. The adjunction formula gives $\omega_C\simeq\mathcal O_C$ for the canonical sheaf $\omega_C$, so $C$ has genus one. The same calculation gives good reduction at $17,19,31$: none divides $2\cdot20\cdot6\cdot2861\cdot1015\cdot37466$.

  To show that $C(\Q)$ is infinite, we compare its rational points with the torsion bound from good reduction. Direct enumeration gives
  \[
    \#C(\F_{17})=16,\qquad \#C(\F_{19})=24,\qquad \#C(\F_{31})=28.
  \]
  These small counts can be obtained by summing $r_p(20u^2+2861)r_p(1015-6u^2)$ over $u\in\F_p$, then adding $r_p(20)r_p(-6)$ for the points at infinity, where $r_p(b)=\#\{z\in\F_p:z^2=b\}$. The affine counts are $16,24,24$, and the contributions at infinity are $0,0,4$.

  The eight choices of signs in
  \[
    [U:V:W:Z]=[\pm59:\pm1679:\pm977:31]
  \]
  give eight distinct rational points. Choose one as origin, identifying $C$ with its Jacobian. By properness, this point extends to a section of each smooth model at the stated primes, giving good reduction for the resulting elliptic curve. Prime-to-$p$ torsion injects under good reduction \cite[Proposition VII.3.1(b)]{Silverman2009}. Reduction at $17$ excludes every odd primary torsion subgroup except possibly the $17$-primary subgroup; reduction at $31$ excludes that subgroup as well. The two-primary torsion injects at both primes, so its order divides $\gcd(16,28)=4$. At least one of the eight rational points therefore has infinite order, and $C(\Q)$ is infinite.

  We next prove independence by specializing at $c_0=[-59:1679:977:31]$ with $a_0=-5556/961$. The five points become
  \[
    (0,1),\quad(-1,1),\quad(-3/4,-59/248),\quad
    (-15/16,1679/1984),\quad(8,977/31).
  \]
  The cubic $f_{a_0}(X)$ reduces modulo two to $X^3+X^2+1$, so it is irreducible. Let $\theta$ be a root. To certify independence of the descent values, use their squareclass representatives
  \[
    \theta,\quad\theta-1,\quad4\theta-3,\quad16\theta-15,\quad\theta+8.
  \]
  For $(p,r)=(3,2),(5,4),(7,5),(11,4),(13,2)$, the polynomial $f_{a_0}(X)$ lies in $\Z_p[X]$, and Hensel's lemma lifts its simple root $r$ modulo $p$ to a root in $\Q_p$. The resulting embedding $\Q(\theta)\hookrightarrow\Q_p$ defines a prime $\mathfrak p$ of residue degree one at which $\theta\equiv r\pmod{\mathfrak p}$. The five displayed elements have nonzero residues there. Record a square residue as $0$ and a nonsquare as $1$. With rows ordered by the pairs $(p,r)$ and columns by the five displayed elements, this gives
  \begin{equation}\label{eq:rank-five-residue-matrix}
    \begin{pmatrix}
      1 & 0 & 1 & 1 & 0 \\
      0 & 1 & 1 & 0 & 1 \\
      1 & 0 & 1 & 0 & 1 \\
      0 & 0 & 1 & 0 & 0 \\
      1 & 0 & 1 & 0 & 0
    \end{pmatrix},
  \end{equation}
  whose determinant is $1$ in $\F_2$. If a product of these elements were a square in $\Q(\theta)$, its nonzero residue at each chosen prime would be a square. Invertibility of the matrix therefore excludes all $31$ nonempty squareclass products. The descent argument in the proof of Theorem~\ref{thm:four-points} proves independence modulo torsion of the five specialized points. This argument uses only irreducibility and nonsingularity, so it applies at the rational parameter $a_0$.

  Since $a$ is regular at $c_0$ and the discriminant is a unit in $\mathcal O_{C,c_0}$, the projective Weierstrass equation defines an elliptic scheme over this local ring. Properness extends every $\Q(C)$-point over the local ring, and specialization is a group homomorphism. A nontrivial integral relation modulo torsion over $\Q(C)$ would specialize to such a relation in the fibre at $c_0$. Thus the five sections are independent over $\Q(C)$.

  To prove saturation at two, put $F=\Q(C)$ and $\mathcal R=\mathcal O_{C,c_0}$. This local ring is an integrally closed discrete valuation ring with residue field $\Q$, and $a$ is regular at $c_0$. A root of the monic polynomial $f_a\in\mathcal R[X]$ in $F$ would belong to $\mathcal R$ and reduce to a rational root of $f_{a_0}$. Hence $E_a(F)[2]=0$. Denote the five sections by $P_1,\ldots,P_5$. If $2R=\sum_i n_iP_i$ in $E_a(F)$, specialization at $c_0$ and the invertible residue matrix force every $n_i$ to be even. The point $R-\sum_i(n_i/2)P_i$ is then two-torsion and vanishes. This proves saturation at two.

  Finally, we specialize to infinitely many fibres. For fixed $u$, the two equations defining $C$ allow at most four affine pairs $(v,w)$, so $u$ is nonconstant on $C$. Hence $a=(u^2-73)/12$ is nonconstant, and the nonconstant rational function $j(E_a)$ displayed in the proof of Theorem~\ref{thm:pell} remains nonconstant on $C$. The specialization theorem \cite[Theorem C.20.3]{Silverman2009} therefore makes specialization injective at all but finitely many points of $C(\Q)$. The five independent sections give rank at least five in infinitely many fibres over points of $C(\Q)$. The map $j:C\to\mathbb P^1$ has finite fibres, so these include infinitely many curves that are pairwise nonisomorphic over $\overline{\Q}$.
\end{proof}

\begin{remark}\label{rem:rank-five-integrality}
  The construction in Proposition~\ref{prop:rank-five-rational} has exactly one integral value of $a$. If $a\in\Z$, then $u^2=12a+73\geq0$ and $w^2=577-72a\geq0$ give $-6\leq a\leq8$. A rational number whose square is an integer is an integer. The first square condition leaves $a=-6,-4,-2,4,8$; the corresponding values of $v^2$ are $2881,3361,3841,5281,6241$, of which only $6241=79^2$ is a square. Hence $a=8$, realized by $(u,v,w)=(13,79,1)$.
\end{remark}

For integers $a$, the first square condition gives $u^2\equiv1\pmod{12}$. After changing the sign of $u$ if necessary, write $u=6s+1$; then $a=3s^2+s-6$. This raises the question whether there are infinitely many integers $s$ with $3s^2+s-6\geq3$ for which $\rank E_{3s^2+s-6}(\Q)\geq5$.



\bibliographystyle{alpha}
\begingroup
\raggedright
\bibliography{ennola}
\endgroup
\end{document}